\documentclass[final]{siamart0216}
\newsiamremark{remark}{Remark}

\usepackage{amsfonts,amsmath,amssymb}
\usepackage{mathrsfs,mathtools,stmaryrd,wasysym}
\usepackage{enumerate}
\usepackage{xspace,mydef} 
\usepackage{esint}
\usepackage{hyperref}
\DeclareMathAlphabet{\mathpzc}{OT1}{pzc}{m}{it}

\newcommand{\TheTitle}{Analysis and approximation of elliptic problems with Uhlenbeck structure in convex polytopes}
\newcommand{\ShortTitle}{Uhlenbeck structure in convex polytopes}
\newcommand{\TheAuthors}{T.~Mengesha, E.~Ot\'arola, A. J.~Salgado}

\headers{\ShortTitle}{\TheAuthors}

\title{{\TheTitle}\thanks{TM has been partially supported by NSF grants DMS-1910180 and DMS-2206252. EO has been partially supported by ANID through FONDECYT project 1220156. AJS has been partially supported by NSF grant DMS-2111228.}}

\author{Tadele Mengesha\thanks{Department of Mathematics, University of Tennessee, Knoxville, TN 37996, USA.
    (\email{megnesha@utk.edu}, \url{https://math.utk.edu/people/tadele-mengesha/})}
  \and
  Enrique Ot\'arola\thanks{Departamento de Matem\'atica, Universidad T\'ecnica Federico Santa Mar\'ia, Valpara\'iso, Chile.
    (\email{enrique.otarola@usm.cl}, \url{http://eotarola.mat.utfsm.cl/}).}
  \and
  Abner J.~Salgado\thanks{Department of Mathematics, University of Tennessee, Knoxville, TN 37996, USA.
    (\email{asalgad1@utk.edu}, \url{https://math.utk.edu/people/abner-salgado/})}
}

\ifpdf
\hypersetup{
  pdftitle={\TheTitle},
  pdfauthor={\TheAuthors}
}
\fi

\begin{document}

\maketitle

\begin{abstract}
We prove the well posedness in weighted Sobolev spaces of certain linear and nonlinear elliptic boundary value problems posed on convex domains and under singular forcing. It is assumed that the weights belong to the Muckenhoupt class $A_p$ with $p \in (1,\infty$). We also propose and analyze a convergent finite element discretization for the nonlinear elliptic boundary value problems mentioned above. As an instrumental result, we prove that the discretization of certain linear problems are well posed in weighted spaces.
\end{abstract}



\section{Introduction}
\label{sec:intro}

In this paper we intend to provide an {\em analysis} and obtain {\em approximation results} for problems of the following form
\begin{equation}
\label{eq:Diening}
  -\DIV \ba(x,\GRAD u) = -\DIV \bef + g \text{ in } \Omega, \qquad u = 0 \text{ on } \partial \Omega.
\end{equation}
Here, for $d \in \{2,3\}$, $\Omega \subset \R^d$ is a convex polytope, $\ba: \Omega \times \R^d \to \R^d$ is a strongly monotone and coercive vector field which, in addition, is \emph{linear at infinity}; see Section~\ref{sub:Elliptic} below for a detailed description.  For a given measurable data ${\bf f}:\Omega\to \mathbb{R}^{d}$ and $g:\Omega\to \mathbb{R}$, by a solution to \eqref{eq:Diening} we mean any $u \in W^{1,q}_{0}(\Omega),$ for some $q>1$, such that
\begin{equation}\label{eq:Diening-sol}
  \int_{\Omega}\ba(x,\GRAD u)\cdot \GRAD \psi \diff x = \int_{\Omega} \bef \cdot \GRAD \psi \diff x + \int_{\Omega} g \psi \diff x
  \quad \forall \psi\in C_0^{\infty}(\Omega).
\end{equation}

The structural assumptions we make for the nonlinearity $\ba$ allow equation \eqref{eq:Diening-sol} to make sense even if $|\nabla u|$, $|{\bf f}|$, and $g$ lie in $L^{1}_{\text{loc}}(\Omega)$. However, it is well known that uniqueness of solutions for problem \eqref{eq:Diening} in $W^{1,1}_{0}(\Omega)$ is not guaranteed, even for linear equations with continuous coefficients; see \cite{JIN2009773,ASNSP_1964_3_18_3_385_0,MR2548032} for the existence of pathological solutions. To work with a well posed problem, we need the additional assumption that the solutions lie in  $W^{1,q}_{0}(\Omega)$, for some $q>1$.  For $q=2$, it can be shown that equation \eqref{eq:Diening-sol} has a unique solution $u\in W^{1, 2}_0(\Omega)$ with data $|{\bf f}|, g\in L^{2}(\Omega)$ under suitable monotonicity and coercivity properties for the nonlinearity ${\bf a}$; the main technique is the Browder and Minty's nonvariational method for monotone operators \cite[Section 9.1]{evans10}, \cite[Chapter 2]{MR3014456}. For the applicability of the method, it is important that the right-hand side of \eqref{eq:Diening} belongs to the dual of the solution space. Here, we are interested in the general scenario where the data might go beyond the natural duality class that aligns with the operator associated with the equation \eqref{eq:Diening}. Specifically, we assume that the data ${\bf f}$ and $g$ are rough and belong to the weighted spaces $\bL^p(\omega,\Omega)$ and $L^p(\omega,\Omega)$, respectively, for $p \in (1,\infty)$ and $\omega \in A_p$---the so-called Muckenhoupt class; see Section~\ref{sec:Notation} for notation.  As we will review shortly, functions in $W^{1,p}(\omega,\Omega)$ belong to $L^{1+\epsilon}(\Omega)$ for $\epsilon>0$ small, but not necessarily to any higher $L^{p}(\Omega)$--classes. In this case, the solution to \eqref{eq:Diening} is expected to belong in  no better space than $W^{1,q}_{0}(\Omega)$, for $1<q<2$, which makes the well posedness of equation \eqref{eq:Diening} a more difficult problem.  We refer the reader to \cite{MR3531368} for a detailed literature review and discussion on the challenges in analyzing the quasilinear equation \eqref{eq:Diening}, which can lead to very weak solutions, as they are commonly called.

This work is motivated by the results obtained in \cite{MR3531368} on the existence, uniqueness, and optimal regularity estimates of solutions for \eqref{eq:Diening} and has two main goals. The first is to prove the existence of a unique solution $u \in W^{1, p}_{0}(\omega, \Omega)$ to \eqref{eq:Diening} corresponding to data $|{\bf f}|$ and $g$ in the weighted class $L^p(\omega,\Omega)$, for any $p>1$, when the domain $\Omega$ is a {\em convex polytope}.  While elliptic problems of type \eqref{eq:Diening} have already been analyzed in various forms in the literature, which is too extensive to cite here, the applicability of the results is limited to problems over domains with either $C^1$ boundary \cite{MR3531368} or a boundary with small Lipschitz constant \cite{Adimurthi2021, BF03377367, byun2023existence}, or a Riefenberg flat boundary \cite{MENGESHA20112485}. The argument used in the aforementioned papers to obtain global estimates requires first deriving local interior and boundary estimates for solutions. The crucial near-boundary estimates in turn require local estimates for flat domains, which are then extended to the domain of interest via a flattening argument. This explains the necessity of the flatness of the boundary. However, this strategy falls short when dealing with problems in convex domains with large Lipschitz constants. Even when dealing with convex domains with small Lipschitz constants, the results of these works are not fully satisfactory due to the restriction on the range of $p$ as a function of the Lipschitz constant. In contrast, our work introduces a different approach to obtain global estimates for solutions on an arbitrary convex domain; the results can be applied for any $p>1$.

Our interest in working with convex polytopes is related to the second goal of the paper, namely the development of a finite element approximation theory for problem \eqref{eq:Diening}. Since the domain $\Omega$ is a convex polytope, it can be meshed exactly. We propose and analyze a convergent finite element discretization for problem \eqref{eq:Diening}. More precisely, we will prove that a discretization of \eqref{eq:Diening} that seeks for a solution in the finite element space of continuous piecewise linear functions defined over a quasiuniform mesh of the domain is convergent with respect to the weak topology of the weighed Sobolev space  $W^{1, p}_{0}(\omega, \Omega)$. The point here is that the approach we use allows us to approximate problems that assume the framework of \emph{very weak} solutions and, more generally, the method is applicable to any convex domain and any $p>1$ and $\omega\in A_p$.

Following the argumentation in \cite{MR3531368} (see also \cite{byun2023existence}) we show the existence of a solution $u$ to problem \eqref{eq:Diening} via an approximation argument. In this approach, the solution is constructed as the limit of a sequence of solutions for the same equation, but with smooth data that converges to the original data. Performing this procedure requires us that we first obtain an \emph{a priori} estimate for any solution $u \in W^{1,p}_{0}(\omega, \Omega)$ in terms of the $L^{p}(\omega, \Omega)$-norms of $|{\bf f}|$ and $g$. To achieve this, we rewrite equation \eqref{eq:Diening} as the following linear equation:
\begin{equation}\label{LE-introduction}
  -\DIV (\bA(x)\GRAD u) = \mathcal{F} \text{ in } \Omega, \qquad u = 0 \text{ on } \partial \Omega,
\end{equation}
where $\mathcal{F}=-\DIV \bef + g + \DIV ( \ba(x, \GRAD u) - \bA(x)\GRAD u)$.  By choosing $\bA(x)$ as a continuous matrix such that  $\bA(x)\bv$ is uniformly close to $\ba(x, \bv)$ at infinity (for large $\bv$), we will show that the right hand side $\mathcal{F}$ lies in the dual space of $W^{1,p'}_{0}(\omega', \Omega)$.  The problem of obtaining \emph{a priori} estimates for solutions to the nonlinear problem is then reduced to the search of optimal regularity estimates for solutions to linear equations of type \eqref{LE-introduction} over convex domains. This is the step where our approach differs from that of \cite{MR3531368}, which obtains the needed optimal global regularity estimates for linear problems  posed over $C^1$ domains. As already mentioned, this is not applicable for our purposes. Instead, we extend a result proved in \cite{MR1198129} about the solvability of linear equations of type \eqref{LE-introduction} on convex domains over unweighted spaces to the case of  weighted spaces. We overcome the problem related to the near-boundary estimates by using available estimates over convex domains for suitable Green's functions. Once we establish this result, the remaining steps follow the reasoning from \cite{MR3531368}.

The paper is structured as follows. We begin with Section \ref{sec:Notation}, where we lay out most of the notations used in this paper. We also define the class of Muckenhoupt 
$A_p$ weights and discuss some of the properties that are needed in this paper. In addition, this section contains all the structural conditions for the nonlinearity ${\bf a}$ that specify the type of nonlinear problems we aim to investigate. In Section \ref{sub:Elliptic}, we establish the solvability of linear elliptic problems with continuous coefficients on convex polytopes. These foundations play a role in Section \ref{sec:Quasilinear}, where we sketch the procedure for proving the existence and uniqueness of solutions to \eqref{eq:Diening}.  We conclude the paper by proposing and analyzing a convergent finite element discretization for problem \eqref{eq:Diening}.

\section{Notations, preliminaries, and structural assumptions}
\label{sec:Notation}

\subsection{Notations}
The first notation we establish is the relation $A \lesssim B$. This shall mean that $A \leq C B$ for a nonessential constant $C$, which can change at each occurrence. $A \gtrsim B$ means $B \lesssim A$, and $A \eqsim B$ is the short form for $A \lesssim B \lesssim A$. If it is necessary to explicitly mention a constant $C$, we assume that $C>0$ and that the value can change every time it occurs.

Throughout our work, $d \in \{2,3\}$ is the spatial dimension and $\Omega \subset \R^d$ is a bounded and convex polytope. The restriction $d \in \{2,3\}$ is only due to the fact that many of the results required for our purposes, for instance those related to Green's functions on convex polytopes, are only known up to dimension three ($d=3$). As soon as these become available for higher dimensions, our results will follow with little or no change. If $z \in \R^d$ and $r>0$, we denote by $\mathring{B}_r(z)$ the (open) Euclidean ball with center at $z$ and radius $r$. Its closure is $B_r(z)$. If $\bA, \bB \in \R^{d \times d}$ are symmetric, we mean by $\bA \preceq \bB$ an order in the spectral sense, \ie
\[
  \bA \bxi \cdot \bxi \leq \bB \bxi \cdot \bxi \qquad \forall \bxi \in \R^d.
\]

Vector valued functions are written in script boldface, and their spaces are also written in boldface characters. Matrix valued functions are denoted with capital boldface letters, and their spaces are written in calligraphic boldface. For example, $\bw \in \bC(\bar\Omega)$ denotes a continuous vector valued function $\bw : \bar\Omega \to \R^d$; while $\bA \in \calbC(\bar\Omega)$ means that $\bA : \bar\Omega \to \R^{d \times d}$ is continuous. If the range of a matrix valued function is contained in the space of symmetric matrices $\R^{d \times d}_\sym$, we denote it by, say, $\calbC_\sym(\bar\Omega)$. 

Let $K \subset \R^d$ be compact. We recall that Besicovitch's covering theorem (see \cite[Theorem 1.3.5]{MR1014685}, \cite[Theorem 1.27]{MR3409135}, and \cite[Theorem 1.4.6]{MR2790542}) guarantees that from any open cover $\{\calO_m\}_{m\in \calM}$ of $K$ we can extract a finite subcover with a \emph{finite overlapping property}. We shall quantify this by saying that if $\{\calO_m\}_{m=1}^M$ is such a finite subcover satisfying $\max\{ \diam\calO_m : m \in \{1,\dots,M\} \} \leq \delta$ and $\{ \varphi_m \}_{m=1}^M \subset C_0^\infty(\R^d)$ is a partition of unity subject to $\{\calO_m\}_{m=1}^M$, then we have
\begin{equation}
\| D^k \varphi_m \|_{\bL^\infty(\R^d;\R^{d^k})} \leq C_{\varphi,k} \delta^{-k},
\quad
\sup_{x \in \bar\Omega} \#\left\{ m \in \{ 1, \ldots, M\} : x \in \supp \varphi_m \right\} \leq \mathcal{N},
\label{eq:scaling_and_Besicovitch}
\end{equation}
where $\mathcal N$ depends only on $d$.

\subsection{$A_p$ weights}
Whenever $w : \R^d \to \R$ is locally integrable with respect to the Lebesgue measure and $E \subset \R^d$ has a positive and finite measure, we set
\[
  \fint_E w(x) \diff x = \frac1{|E|} \int_E w(x) \diff x.
\]
By a weight we mean a locally integrable function $\omega : \R^d \to \R$ which is positive almost everywhere. Given a weight $\omega$ and $p \in [1,\infty)$, we set $L^p(\omega,\Omega)$ to be the Lebesgue space of integrability index $p$ on $\Omega$ with respect to the measure $\omega \diff x$, \ie
\[
  L^p(\omega,\Omega) = \left\{ f : \Omega \to \R : \| f \|_{L^p(\omega,\Omega)} = \left( \int_\Omega |f(x)|^p \omega(x) \diff x \right)^{\frac{1}{p}} < \infty \right \}.
\]

For $p \in [1,\infty)$ we say that a weight $\omega$ belongs to the Muckenhoupt class $A_p$ if there is a constant $C>0$ such that for every ball $B \subset \R^d$ \cite[Definition 1.2.2]{MR1774162}
\begin{align*}
  \left( \fint_B \omega(x) \diff x \right) \left( \fint_B \omega(x)^{-\tfrac1{p-1}} \diff x \right)^{p-1} \leq C, && p &> 1,
  \\
  \left( \fint_B \omega(x) \diff x \right) \sup_{x \in B} \frac{1}{\omega(x)} \leq C, && p &= 1.
\end{align*}
We call the best constant in the previous inequalities the Muckenhoupt characteristic of $\omega$ and denote it as $[\omega]_{A_p}$. In addition, we define $A_{\infty} = \cup_{p\geq1} A_p$. Note that, 
\[
  \omega \in A_p \iff \omega' = \omega^{-\tfrac1{p-1}} \in A_{p'}, 
 \qquad
 [\omega]_{A_p} = [\omega']_{A_{p'}}^{p-1};
\]
for $p \in (1,\infty)$ \cite[Remark 1.2.4]{MR1774162}, where we denoted the H\"older conjugate of $p \in (1,\infty)$ by $p'$. Many useful properties follow from the fact that $\omega \in A_p$ ($p \in [1,\infty)$). We mention here those that will be useful for us in the following:
\begin{enumerate}[(i)]
  \item $L^p(\omega,\Omega) \subset L^1_{\text{loc}}(\Omega)$ \cite[(1.2.1)--(1.2.2)]{MR1774162}. In particular,  elements of $L^p(\omega,\Omega)$ are distributions and we can talk about their distributional derivatives.
  
  \item \emph{Reverse H\"older inequality}: There exist constants $\vare >0$ and $C>0$, depending only on $p$ and $[\omega]_{A_p}$, such that for any ball $B \subset \mathbb{R}^d$,
  \begin{equation}
    \left( \fint_B \omega(x)^{1+\vare} \diff x \right)^{1/(1+\vare)} \leq C \fint_B \omega(x) \diff x;
  \label{eq:reverse_Holder_inequality}
  \end{equation}
  see \cite[Theorem 7.4]{MR1800316} and \cite[Lemma 1.2.12]{MR1774162}. From this it first follows that $L^p(\omega,\Omega) \subset L^{1+\epsilon}(\Omega)$ for a sufficiently small $\epsilon >0$. Furthermore, we can conclude that for $p \in (1,\infty)$ and $\omega \in A_p$, a function $v$ belongs to $L^p(\omega,\Omega)$ provided that $v \in L^q(\Omega)$ for $q = p(1+\vare)/\vare > 1$; $\vare$ is such that \eqref{eq:reverse_Holder_inequality} holds.
  
  \item \emph{Embedding}: If $1 \leq p < q < \infty$, then $A_p \subset A_q$; see \cite[Remark 1.2.4]{MR1774162}.
  
  \item \emph{Open ended property}: If $\omega \in A_p$ with $p\in (1,\infty)$, then there is $\delta >0$ such that $\omega \in A_{p-\delta}$; see \cite[Corollary 1.2.17]{MR1774162} and \cite[Corollary 7.6, item (2)]{MR1800316}.
  
  \item \emph{Lattice property}: If $\omega_1, \omega_2 \in A_p$, then $\min\{ \omega_1, \omega_2 \}$ and $\max\{\omega_1,\omega_2\}$ belong to $A_p$.
  
  \item \label{itemvi}
  \emph{Factorization}: Let $w \in L^1_{\text{loc}}(\R^d)$ be a nonnegative function that satisfies $\calM[w] < \infty$ almost everywhere and let $k$ be a nonnegative function such that $k, k^{-1} \in L^{\infty}(\mathbb{R}^d)$. Then, for every $\vare \in (0,1)$,
   \[
    \omega(x) = k(x) \calM[w](x)^\vare
  \]
  is an $A_1$ weight; see \cite[Theorem 7.2.7]{MR3243734} and also \cite[Theorem 7.7]{MR1800316}. Here, $\calM$ denotes the Hardy--Littlewood maximal function. Moreover, for all $p \in (1,\infty)$ and all $\alpha \in (0,1)$, there holds $\calM[w]^{-\alpha(p-1)} \in A_p$ \cite[Lemma 3.2]{MR3531368}.
\end{enumerate}
We refer the reader to \cite{MR1774162,MR1800316,MR2305115,MR3243734} for more details on Muckenhoupt weights and weighted-norm inequalities.

As already mentioned, for any $p \in [1,\infty)$ and $\omega \in A_p$, the elements of $L^p(\omega,\Omega)$ are distributions. Thus, it makes sense to talk about their distributional derivatives and define weighted Sobolev spaces:
\begin{align*}
   W^{1,p}(\omega,\Omega) &= \left\{ w \in L^p(\omega,\Omega) : \partial_i w \in L^p(\omega,\Omega) \, \forall i \in \{ 1, \ldots, d \} \right\},
\end{align*}
endowed with $\| w \|_{W^{1,p}(\omega,\Omega)} = \left( \| w \|_{L^{p}(\omega,\Omega)} + \| \GRAD w \|_{\bL^{p}(\omega,\Omega)} \right)^{\frac{1}{p}}$. We define $W^{1,p}_0(\omega,\Omega)$ as the closure of $C_0^{\infty}(\Omega)$ in $W^{1,p}(\omega,\Omega)$. Due to the fact that the weight $\omega \in A_p$ many of the properties of the classical Sobolev spaces extend to the weighted ones \cite{MR1774162,MR2491902,MR0802206}. In particular, we have a weighted Poincar\'e inequality: if $D \subset \R^d$ is open and bounded, $p \in (1,\infty)$, and $\omega \in A_p$, then \cite[Theorem 1.3]{MR0643158}
\begin{equation}
\label{eq:wPoincare}
  \| w \|_{L^p(\omega,D)} \leq {C_p} \diam (D) \| \GRAD w \|_{\bL^p(\omega,\Omega)} \quad \forall w\in W^{1,p}_0(\omega,\Omega).
\end{equation}
The constant {$C_p$} depends on $\omega$ only through $[\omega]_{A_p}$. Finally, the dual of $W^{1,p'}_0(\omega',\Omega)$ will be denoted by $W^{-1,p}(\omega,\Omega)$.

\subsection{Assumptions on the nonlinearity}

We now turn our attention to \eqref{eq:Diening}. Our goal is to find a set of conditions for the nonlinearity and the variable coefficient that guarantee the existence of a unique solution, together with a weighted stability estimate. For this purpose,
we shall assume that $\ba : \Omega \times \R^d \to \R^d$ satisfies the following conditions.
\begin{enumerate}[$(A)$]
  \item \label{item:A}
  \emph{Carath\'eodory mapping}: $\ba : \Omega \times \R^d \to \R^d$ is Carath\'eodory. In other words $\ba(\cdot,\bv)$ is measurable for any fixed $\bv \in \mathbb{R}^d$ and $\ba(x,\cdot)$ is continuous for almost all $x \in \Omega$.

  \item \emph{Coercivity}: There is $\alpha >0$ such that for almost all $x \in \Omega$ and every $\bv \in \R^d$,
  \begin{equation}
  \label{eq:baiscoercive}
    \alpha |\bv|^2 \leq \ba(x,\bv) \cdot \bv.
  \end{equation}
  
  \item \emph{Growth}: There is $\Lambda >0$ such that for almost all $x \in \Omega$ and all $\bv \in \R^d$,
  \begin{equation}
  \label{eq:baisgrowth}
    |\ba(x,\bv)| \leq \Lambda |\bv|.
  \end{equation}

  \item \emph{Strict monotonicity}: For almost all $x \in \Omega$ and every $\bv, \bw \in \R^d$ such that $\bv \neq \bw$,
  \begin{equation}
  \label{eq:baismonotone}
    \left( \ba(x,\bv) - \ba(x,\bw) \right) \cdot (\bv - \bw ) >0.
  \end{equation}
  
  \item \emph{$\ba$ is asymptotically Uhlenbeck}: There exists $\bA \in \calbC_\sym(\bar{\Omega})$ 
  such that, for all $\vare>0$, there is $N>0$, such that for almost all $x\in \Omega$ and all $\bv \in \R^d$ satisfying $|\bv| \geq N$, we have
  \begin{equation}
  \label{eq:baisUhlenbeck}
    \left|\ba(x,\bv) - \bA(x)\bv\right| \leq \vare |\bv|.
  \end{equation}
  
  \item
  \label{item:F}
  \emph{$\ba$ is strongly asymptotically Uhlenbeck}: There exists $\bA \in \calbC_\sym( \bar{\Omega})$ such that, for all $\vare>0$, there is $N>0$, such that for almost all $x\in \Omega$ and all $\bv \in \R^d$ satisfying $|\bv| \geq N$, we have
  \begin{equation}
      \left| \partial_\bv \ba(x,\bv) - \bA(x) \right| \leq \vare.
    \end{equation}
\end{enumerate}

\begin{remark}[structure conditions]
A prototypical example of a mapping $\ba$ that satisfies the asymptotically Uhlenbeck structure condition is
\[
  \ba(x,\bv) = a(x,|\bv|) \bv: \quad \lim_{t \uparrow \infty} a(x,t) = \tilde{a}(x) \mae x \in \Omega,
\]
with $\tilde{a} \in C(\bar\Omega)$; see the discussion at the beginning of Section 2 in \cite{MR3531368}. Note that the coefficient $a$ is allowed to be merely measurable in $x$ and the required continuity must hold only for $\tilde a$. The mapping $\ba$ is strongly asymptotically Uhlenbeck if $a$ is differentiable with respect to $t$ for $t \gg 1$ and $|\partial_t a(x,t) t | \rightarrow 0$ as $t \uparrow \infty$. We refer the reader to \cite{MR3531368} and our introduction for further motivation.
\end{remark}

\section{Linear elliptic problems with continuous coefficients}
\label{sub:Elliptic}
We begin the analysis of \eqref{eq:Diening} with the linear case, \ie
\begin{equation}
\label{eq:Aislinear}
  \begin{dcases}
    \ba(x,\bv) = \bA(x)\bv & \forall (x,\bv) \in \bar\Omega\times \R^d, \\
    \bA \in \calbC_\sym(\bar\Omega), & \| \bA \|_{\calbC(\bar\Omega)} \leq \Lambda, \\
    \exists \alpha >0 : \alpha \bI \preceq \bA(x) & \forall x \in \bar\Omega.
  \end{dcases}
\end{equation}

As a first preparatory step, let us recall some known results. We begin with the simplest variant of \eqref{eq:Diening}, \ie that of the Laplacian: $\ba(x,\bv) = \bv$. This particular case is studied in \cite{DDO:17}. In particular, the following proposition is proved in \cite[Corollary 2.7]{DDO:17}; see also \cite{MR1156467}.

\begin{proposition}[weighted stability]
\label{prop:weightedLaplace}
Let $D \subset \R^d$ be bounded and convex. Let $p \in (1,\infty)$ and $\omega \in A_p$. For every $\calF \in W^{-1,p}(\omega,D)$, there is a unique $U \in W^{1,p}_0(\omega,D)$ such that $-\Delta U = \calF$ in $\mathscr{D}'(D)$. Moreover,
\[
  \|\GRAD U \|_{\bL^p(\omega,D)} \leq C_\Delta \| \calF \|_{W^{-1,p}(\omega,D)},
\]
where $C_\Delta$ does not depend on $D$ and depends on $\omega$ only through $[\omega]_{A_p}$.
\end{proposition}

\label{sec:TheLinearPb}

The following is the the most important result of this section, which extends the solvability of the linear elliptic equation obtained in \cite{MR1198129} to the weighted case.

\begin{theorem}[well posedness]
\label{thm:LinearWellPosed}
Let $d \in \{2,3\}$ and let $\Omega \subset \R^d$ be a bounded and convex polytope. Assume that $\ba$ satisfies \eqref{eq:Aislinear}. Let $p \in (1,\infty)$ and let $\omega \in A_p$. Then, for any $\bef \in \bL^p(\omega,\Omega)$ and $g \in L^p(\omega,\Omega)$, problem \eqref{eq:Diening} has a unique solution $u \in W^{1,p}_0(\omega,\Omega)$. Moreover, $u$ satisfies
\begin{equation}
\label{eq:stabilityLinear}
  \| \GRAD u \|_{\bL^p(\omega,\Omega)} \lesssim \| \bef \|_{\bL^p(\omega,\Omega)} + \| g \|_{L^p(\omega,\Omega)}.
\end{equation}
The hidden constant depends on $\omega$ only through its Muckenhoupt characteristic $[\omega]_{A_p}$ and on the coefficient $\bA$ only through $\alpha$ and $\Lambda$.
\end{theorem}

The proof follows the argumentation in \cite{MR1198129}. We prove an a priori estimate in weighted spaces via the method of localization by freezing the coefficient and using the results of Proposition~\ref{prop:weightedLaplace}. We then prove the existence of a solution as the limit of a sequence of solutions to the same problem, but with smooth data converging to the original data. We have divided the proof into several steps in which we prove several independent statements. We begin with an a priori estimate.

\begin{lemma}[G\aa{}rding--like inequality]
\label{lem:Garding}
Let $d \in \{2,3\}$ and let $\Omega \subset \R^d$ be a bounded and convex polytope. Let $\bB \in \calbC^\infty_\sym(\bar\Omega)$ be such that there is $\alpha >0$ for which
\begin{equation}
\label{eq:alpha}
 \alpha \bI \preceq \bB(x) \quad \forall x \in \bar\Omega.
\end{equation}
Let $p \in (1,\infty)$ and $\omega \in A_p$. Given $\bmu \in \bL^p(\omega,\Omega)$ and $\gamma \in L^p(\omega,\Omega)$, assume that $\fraku \in W^{1,p}_0(\omega,\Omega)$ satisfies
\[
  -\DIV ( \bB \GRAD \fraku )= -\DIV \bmu + \gamma \quad \text{ in } \mathscr{D}'(\Omega).
\]
Then,
\begin{equation}
\label{eq:first_weighted_bound}
  \| \GRAD \fraku \|_{\bL^p(\omega,\Omega)} \lesssim \| \bmu \|_{\bL^p(\omega,\Omega)} + \| \gamma \|_{L^p(\omega,\Omega)} + \| \fraku \|_{L^p(\omega,\Omega)}.
\end{equation}
The hidden constant depends on the weight $\omega$ only through its Muckenhoupt characteristic $[\omega]_{A_p}$ and on the coefficient $\bB$ only through the norm $\|\bB\|_{\calbC(\bar\Omega)}$ and $\alpha$.
\end{lemma}
\begin{proof}
The estimate \eqref{eq:first_weighted_bound} follows from a localization argument by freezing the coefficients. For this purpose, we set $\Lambda = \| \bB \|_{\calbC(\bar\Omega)}$ and let $\vare >0$ be sufficiently small so that
\begin{equation}
\label{eq:ChoiceOfEps}
  4^{p-1} \mathcal{N}^{p+2} C_\Delta^p \Lambda^{\frac{p}{2}} \alpha^{-p-\frac{dp}{2}} (C_{\varphi,1}C_{p'}+1)^p \varepsilon^p< \frac12,
\end{equation}
where $\mathcal{N}$ and $C_{\varphi,1}$ come from \eqref{eq:scaling_and_Besicovitch} and $C_\Delta$ is as in Proposition~\ref{prop:weightedLaplace}.
Since $\bB \in \calbC_\sym(\bar\Omega)$, it is uniformly continuous in $\bar\Omega$. For every $\vare >0$ there is therefore $\delta>0$ such that for any $x_0 \in \bar\Omega$ and all $x \in \bar\Omega\cap B_\delta(x_0)$
\begin{equation}
\label{eq:OscOnBest}
  |\bB(x) - \bB(x_0) | \leq \vare.
\end{equation}
By compactness of $\bar{\Omega}$, we deduce the existence of $M \in \polN$ and $\{x_m\}_{m=1}^M \subset \bar\Omega$ such that
\[
  \bar\Omega \subset \bigcup_{m=1}^M \mathring{B}_\delta(x_m).
\]
For $m \in \{ 1,\ldots,M\}$, we define $\bar\Omega_m := \bar\Omega \cap B_\delta(x_m)$ and $\bB_m = \bB(x_m) \in \R^{d \times d}_\sym$. Notice that, for every $m \in \{ 1,\ldots,M\}$, $\Omega_m$ is a convex domain.

Let $\{\varphi_m\}_{m=1}^M \subset C_0^\infty(\R^d)$ be a partition of unity subordinate to $\{\mathring{B}_\delta(x_m)\}_{m=1}^M$. Without loss of generality, we may assume that $\{\varphi_m\}_{m=1}^M$ satisfies \eqref{eq:scaling_and_Besicovitch}.

Define $\fraku_m = \fraku \varphi_m \in W^{1,p}_0(\omega,\Omega_m)$, and let $v \in C_0^\infty(\Omega_m)$. Note that,
\begin{align*}
  \int_{\Omega_m} \bB_m \GRAD \fraku_m \cdot \GRAD v \diff x &= \int_{\Omega_m} \bB_m \GRAD \fraku \cdot \GRAD(\varphi_m v) \diff x + \int_{\Omega_m} \fraku \bB_m \GRAD v \cdot \GRAD \varphi_m \diff x \\ &- \int_{\Omega_m} v \bB_m \GRAD \fraku \cdot \GRAD \varphi_m \diff x.
\end{align*}
If we integrate by parts the last term on the right hand side of the previous relation, we obtain
\begin{align*}
  - \int_{\Omega_m} v \bB_m \GRAD \fraku \cdot \GRAD \varphi_m \diff x &= \int_{\Omega_m} \fraku \DIV( v \bB_m \GRAD \varphi_m) \diff x \\
   &= \int_{\Omega_m} \fraku \left[ \bB_m \GRAD \varphi_m \cdot \GRAD v + v \bB_m : D^2 \varphi_m \right] \diff x.
\end{align*}
To obtain the last equality, we have used that $\bB_m$ is a constant matrix. Thus,
\begin{align*}
  \int_{\Omega_m} \bB_m \GRAD \fraku_m \cdot \GRAD v \diff x &= \int_{\Omega_m} \bB_m \GRAD \fraku \cdot \GRAD(\varphi_m v) \diff x + 2\int_{\Omega_m} \fraku \bB_m \GRAD v \cdot \GRAD \varphi_m \diff x \\ &+ \int_{\Omega_m} \fraku v \bB_m : D^2 \varphi_m \diff x \\
  & =\int_{\Omega} (\bB_m - \bB) \GRAD \fraku \cdot \GRAD(\varphi_m v) \diff x  + \int_\Omega \bB \GRAD \fraku \cdot \GRAD(\varphi_m v)  \diff x \\
  &+ \int_{\Omega_m} \fraku \left[2\bB_m \GRAD v \cdot \GRAD \varphi_m + v \bB_m : D^2 \varphi_m \right] \diff x.
\end{align*}
In other words, $\fraku_m \in W^{1,p}_0(\omega,\Omega_m)$ satisfies
\begin{equation}
\label{eq:LocProblemConstCoeff}
  \int_{\Omega_m} \bB_m \GRAD \fraku_m \cdot \GRAD v \diff x = \langle \calF_m, v \rangle \qquad \forall v \in C_0^\infty(\Omega_m),
\end{equation}
where the distribution $\calF_m$ is defined as 
\begin{multline}\label{Eqn:forCallF}
  \langle \calF_m, v \rangle = \int_\Omega \bmu \cdot \GRAD(\varphi_m v) \diff x + \int_\Omega \gamma \varphi_m v \diff x + \int_{\Omega} (\bB_m - \bB) \GRAD \fraku \cdot \GRAD(\varphi_m v) \diff x \\
  + \int_{\Omega_m} \fraku \left[2\bB_m \GRAD v \cdot \GRAD \varphi_m + v \bB_m : D^2 \varphi_m \right] \diff x.
\end{multline}
It is not difficult to see that the distribution  $\calF_m$ belongs to $W^{-1,p}(\omega,\Omega_m)$. 
 
We now note that since the matrix $\bB_m$ is symmetric and positive definite, up to an affine transformation the operator in \eqref{eq:LocProblemConstCoeff} is the Laplacian. In fact, $\bB_m = \bQ_m^\intercal \bLambda_m \bQ_m$ with $\bQ_m$ orthogonal and $\bLambda_m$ diagonal. Define the (linear) transformation
\[
  y = \bF_m (x) = \bLambda_m^{{-1/2}} \bQ_m x,
  \quad
  {x \in \Omega_m},
  \quad
  D\bF_m(x) = \bLambda_m^{{-1/2}} \bQ_m.
\]
With the change of variables $\hat{w}(y) = w(x),$  it follows that
$ \GRAD_x w(x) = D\bF_m^\intercal \GRAD_y \hat{w}{(y)}$. Furthermore, $w\in C_0^{\infty}({\Omega_m})$ if and only if $\hat{w} \in C_0^{\infty}(D_m)$, where ${D_m = \bF_m(\Omega_m)}$. It follows that for each $v\in C_0^{\infty}(\Omega)$,
\begin{align*}
  \int_{\Omega_m} \bB_m \GRAD \fraku_m \cdot \GRAD v \diff x
  &
  = \int_{D_m} \GRAD_y \hat{v}^\intercal D\bF_m \bB_m D\bF_m^\intercal \GRAD_y \hat{\fraku}_m |\det D\bF_m|^{-1} \diff y
  \\
  &
  = \left| \det \bLambda_m^{{-\frac{1}{2}}} \right|^{-1} \int_{D_m} \GRAD_y \hat{\fraku}_m \cdot \GRAD_y \hat{v} \diff y.
\end{align*}
Applying a similar change of variables, we consider $\hat{\calF}_m\in W^{-1,p}(\hat{\omega}_{m},D_m)$ {to be such that
% Applying a similar change of variables in \eqref{Eqn:forCallF}, we consider $\hat{\calF}_m\in W^{-1,p}(\hat{\omega}_{m},D_m)$ to be such that
$
  \langle \calF_m, v\rangle = \langle\hat{\calF}_m,\hat{v}\rangle
$
for all $v \in C_0^{\infty}(\Omega_m)$,} where we have set $\hat\omega_m = \omega \circ \bF^{-1}_m$. We now note that
\begin{align*}
  \| \calF_m \|_{W^{-1,p}(\omega,\Omega_m)}
  &
  \leq
  \left |\det \bLambda_m^{{-\frac{1}{2}}} \right |^{-\frac{1}{p}}
  \frac{ \max_{i=1 \ldots d} \bLambda_{m,ii}^{{{-\frac{1}{2}}}} }{  \min_{i=1 \ldots d} \bLambda_{m,ii}^{{-\frac{1}{2}}} } \| \hat{\calF}_m \|_{W^{-1,p}(\hat{\omega}_m,D_m)}, \\
  \| \hat{\calF}_m \|_{W^{-1,p}(\hat\omega_m,D_m)}
  &
  \leq
  \left| \det \bLambda_m^{\frac{1}{2}} \right|^{-\frac{1}{p}} \frac{ \max_{i=1 \ldots d} \bLambda_{m,ii}^{\frac{1}{2}} }{  \min_{i=1 \ldots d} \bLambda_{m,ii}^{\frac{1}{2}}  } \| \calF_m \|_{W^{-1,p}(\omega,\Omega_m)}.
\end{align*}
Combining all of the above calculations, we obtain that $\hat{\fraku}_m$, which now belongs to $W^{1, p}_{0}(\hat{\omega}_m, D_m)$ solves the equation 
\[
-\Delta \hat{\fraku}_m = \left| \det \bLambda_m^{-\frac{1}{2}} \right| \hat{\calF}_m \text{ in } \mathscr{D}'(D_m).
\]

Now, since $D_m$ is convex, by Proposition~\ref{prop:weightedLaplace} we have
\begin{align*}
  \| \GRAD \hat{\fraku}_m \|_{\bL^p(\hat{\omega}_m,D_m)}
  &
  \leq  C_\Delta |\det \bLambda_m^{{-\frac{1}{2}}} | \| \hat{\calF}_m \|_{W^{-1,p}(\hat{\omega}_m,D_m)}\\
  & \leq C_\Delta \left|\det \bLambda_m^{\frac{1}{2}} \right|^{-1-\frac{1}{p}} \frac{ \max_{i=1 \ldots d} \bLambda_{m,ii}^{\frac{1}{2}} }{  \min_{i=1 \ldots d}  \bLambda_{m,ii}^{\frac{1}{2}} } \| \calF_m \|_{W^{-1,p}(\omega,\Omega_m)}.
\end{align*}
A further change of variables then shows that
\begin{align*}
  \int_{D_m} |\GRAD_y \hat{\fraku}_m |^p \hat{\omega}_m \diff y &= \int_{\Omega_m} |D\bF_m^{-\intercal} \GRAD_x \fraku_m|^p  |\det D\bF_m| \omega\diff x \\
    &
    =  \left|\det \bLambda_m^{-\frac{1}{2}} \right| \int_{\Omega_m} |\bLambda_m^{\frac{1}{2}} \bQ_m \GRAD_x \fraku_m|^p \omega \diff x \\
    &
    \geq {\left( \min_{i=1 \ldots d} \bLambda_{m,ii}^{\frac{1}{2}}  \right)^p |\det \bLambda_m^{-\frac{1}{2}}} | \| \GRAD \fraku_m \|_{\bL^p(\omega,\Omega_m)}^p.
\end{align*}
Finally, we use the spectral bounds for $\bB_m$ to conclude that
\begin{align*}
  \| \GRAD \fraku_m \|_{\bL^p(\omega,\Omega_m)}
  &\leq
  {C_\Delta  |\det \bLambda_m^{\frac{1}{2}} |^{-1} \frac{ \max_{i=1 \ldots d} \bLambda_{m,ii}^{\frac{1}{2}} }{  \min_{i=1 \ldots d} \bLambda_{m,ii}  }} \| \calF_m \|_{W^{-1,p}(\omega,\Omega_m)} \\
  &\leq {C_\Delta \Lambda^{\frac{1}{2}} \alpha^{-1 -\frac{d}{2} }} \| \calF_m \|_{W^{-1,p}(\omega,\Omega_m)}.
  %  &\lesssim C_\Delta \frac{ \Lambda^{\tfrac{d}{2p'}}}{\alpha^{\tfrac1{2p}}} \| \calF_m \|_{W^{-1,p}(\omega,\Omega_m)}.
\end{align*}

Let us now estimate each of the terms that comprise $\calF_m$. Based on the scaling properties of $\varphi_m$ in \eqref{eq:scaling_and_Besicovitch} and the weighted Poincar\'e inequality \eqref{eq:wPoincare}, we conclude that
\begin{equation}
\begin{aligned}
 \left| \int_\Omega \bmu \cdot \GRAD(\varphi_m v) \diff x \right|
 & \lesssim \|\bmu \|_{\bL^p(\omega,\Omega_m)} \left[ \| \GRAD v \|_{\bL^{p'}(\omega',\Omega_m)} + \delta^{-1} \| v \|_{L^{p'}(\omega',\Omega_m)} \right] \\
 & \lesssim \|\bmu \|_{\bL^p(\omega,\Omega_m)} \| \GRAD v \|_{\bL^{p'}(\omega',\Omega_m)},
 \label{estimate_mu}
\end{aligned}
\end{equation}
and 
\[
  \left| \int_\Omega \gamma \varphi_m v \diff x \right| \leq \| \gamma \|_{L^p(\omega,\Omega_m)} \| v \|_{L^{p'}(\omega',\Omega_m)} \lesssim \delta \| \gamma \|_{L^p(\omega,\Omega_m)} \| \GRAD v \|_{\bL^{p'}(\omega',\Omega_m)}.
\]
Similarly,
\[
  \left| 2\int_{\Omega_m} \fraku \bB_m \GRAD v \cdot \GRAD \varphi_m  \diff x \right| \lesssim \frac\Lambda\delta \| \fraku \|_{L^p(\omega,\Omega_m)} \| \GRAD v \|_{\bL^{p'}(\omega',\Omega_m)},
\]
and
\begin{align*}
  \left| \int_{\Omega_m} \fraku v \bB_m : D^2 \varphi_m \diff x \right| &\lesssim \frac\Lambda{\delta^2} \| \fraku \|_{L^p(\omega,\Omega_m)} \| v \|_{L^{p'}(\omega',\Omega_m)} \\ &\lesssim \frac\Lambda \delta\| \fraku \|_{L^p(\omega,\Omega_m)} \| \GRAD v \|_{\bL^{p'}(\omega',\Omega_m)},
\end{align*}
where we have again used \eqref{eq:wPoincare} to obtain the last inequality. By the choice of $\delta$, we can therefore estimate
\[
 \left| \int_{\Omega_m} (\bB_m - \bB) \GRAD \fraku \cdot \GRAD (\varphi_m v) \diff x\right|  \leq {(C_{\varphi,1}C_{p'} + 1)} \vare \| \GRAD \fraku \|_{\bL^p(\omega,\Omega_m)} \| \GRAD v \|_{\bL^{p'}(\omega',\Omega_m)},
\]
where $C_{\varphi,1}>0$ is defined in \eqref{eq:scaling_and_Besicovitch} and is independent of $\vare$. Here, we have used the arguments leading to \eqref{estimate_mu} to control $\| \GRAD ( \varphi_m v) \|_{\bL^{p'}(\omega',\Omega_m)}$. To summarize, we have then obtained that
\begin{equation}
\label{eq:bound_for_frau_m}
 \begin{aligned}
  \| \GRAD \fraku_m \|_{\bL^p(\omega,\Omega_m)} &\leq
  { C_\Delta \Lambda^{\frac{1}{2}} \alpha^{-1-\frac{d}{2}} }
    \left[
    C_1
    \left(
    \| \bmu \|_{\bL^p(\omega,\Omega_m)}
    + \delta\| \gamma \|_{L^p(\omega,\Omega_m)}
    \right)
    \right.
    \\
    &+
    \left.
    C_2 \Lambda \delta^{-1} \| \fraku \|_{L^p(\omega,\Omega_m)}
    + {(C_{\varphi,1}C_{p'} + 1)} \vare \| \GRAD \fraku \|_{\bL^p(\omega,\Omega_m)}
    \right].
\end{aligned}
\end{equation}

Having obtained the previous bound we derive an estimate for $\| \GRAD \fraku \|_{\bL^p(\omega,\Omega)}$. Using the finite overlapping property of our covering, it follows that
\[
 \| \GRAD \fraku \|^p_{\bL^p(\omega,\Omega)}
 =
  \int_{\Omega} \omega | \nabla \fraku |^p \mathrm{d}x
 \leq \mathcal{N}  \sum_{m=1}^{M} \int_{\Omega_m} \omega | \nabla \fraku|^p \mathrm{d}x.
\]
We now use the partition of unity property to write $\fraku = \sum_{m}\fraku \varphi_m$ and then use the finite overlapping property again to obtain
\begin{equation}
 \| \GRAD \fraku \|^p_{\bL^p(\omega,\Omega)}
 \leq
 \mathcal{N}  \sum_{m=1}^{M} \int_{\Omega_m} \omega \left| \sum_{i=1}^M \nabla \fraku_i \right|^p \mathrm{d}x
 \leq
 \mathcal{N}^{p+1} \sum_{m=1}^{M} \int_{\Omega_m} \omega | \nabla \fraku_m|^p \mathrm{d}x.
 \label{eq:bound_for_fraku_first}
\end{equation}
To obtain the last bound we have used the following easy consequence of H\"older's inequality:
\[
  \left|\sum_{i=1}^{\mathcal{Z}} z_i \right|^p \leq \mathcal{Z}^{\frac{p}{p'}} \left(\sum_{i=1}^{\mathcal{Z}} |z_i|^p \right), \quad \forall \mathcal{Z} \in \polN, \quad \forall \{z_i\}_{i=1}^{\mathcal{Z}} \subset \R.
\]
This inequality will be used once more below.
We now replace the bound obtained for $\fraku_m$ in \eqref{eq:bound_for_frau_m} into the previously derived bound \eqref{eq:bound_for_fraku_first} to deduce
\begin{multline}
  \| \GRAD \fraku \|^p_{\bL^p(\omega,\Omega)}
  \leq
  4^{p-1}
  \mathcal{N}^{p+2}
  { C_\Delta^p \Lambda^{\frac{p}{2}} \alpha^{-p-\frac{dp}{2}} }
  \left[
    C_1^p\| \bmu \|^p_{\bL^p(\omega,\Omega)} + C_1^p \delta^p \| \gamma \|^p_{L^p(\omega,\Omega)}
  \right. \\
  +
  \left.
    C_2^p \Lambda^{p} \delta^{-p} \| \fraku \|_{L^p(\omega,\Omega)}^p
    + {(C_{\varphi,1}C_{p'} + 1)^p} \varepsilon^p \| \GRAD \fraku \|^p_{\bL^p(\omega,\Omega)}
  \right].
\label{eq:bound_ultimate_fraku}
%  4^{p-1} \mathcal{N}^{p+1} C_\Delta^p \Lambda^{d(p-1)} \alpha C_{\varphi,1}^p \varepsilon^p
\end{multline}
It is at this point that the initially chosen value of $\varepsilon$ becomes relevant; see \eqref{eq:ChoiceOfEps}.
The last term on the right hand side of the previous inequality can be absorbed on the left hand side, and this completes the proof.
\end{proof}

The following a priori estimate follows from this result.

\begin{corollary}[a priori estimate]
Under the conditions of Lemma~\ref{lem:Garding} and assuming that $\bmu$ and $\gamma$ are smooth, we have
\begin{equation}
\label{eq:aprioriForSmooth}
  \| \GRAD \fraku \|_{\bL^p(\omega,\Omega)} \lesssim \| \bmu \|_{\bL^p(\omega,\Omega)} + \| \gamma \|_{L^p(\omega,\Omega)}.
\end{equation}
The hidden constant depends on the weight $\omega$ only through $[\omega]_{A_p}$ and on the coefficient $\bB$ only through $\alpha$ and the norm $\|\bB \|_{\calbC(\bar\Omega)}$.
\label{co:aprioriForSmooth}
\end{corollary}
\begin{proof}
To simplify the notation, we let $\calL \phi = -\DIV(\bB \GRAD \phi)$. For $\Psi \in L^{p'}(\omega',\Omega)$ and smooth, we consider the following auxiliary problem:
\begin{equation}
w \in W^{1,p'}_0(\omega',\Omega):
\quad
\calL^* w = \Psi \textrm{ in } \Omega,
\quad
w = 0 \textrm{ on } \partial \Omega,
\label{eq:problem_for_Psi}
\end{equation}
where $\calL^*$ denotes the formal adjoint of $\calL$. Since $\bB$ and $\Psi$ are assumed to be smooth and $\Omega$ convex, there is a unique solution $w$ to problem \eqref{eq:problem_for_Psi} that belongs to $W^{1,2}_0(\Omega)\cap W^{2,2}(\Omega)$; see \cite[Theorem 3.2.1.2]{MR3396210}. On the other hand, it follows from \cite[Theorem 1]{MR1198129} that $w \in W^{1,q}_0(\Omega)$ for every $q \in (1,\infty)$. Due to the reverse H\"older inequality, it is therefore possible to find a $q$ that is sufficiently large so that $w \in W^{1,q}_0(\Omega) \subset W^{1,p'}_0(\omega',\Omega)$.

We now \emph{claim} that the following regularity estimate is valid:
%Assume that for all smooth $\Psi$ we have shown
\begin{equation}
\label{eq:LqAuxEstimate}
  \| w \|_{L^{p'}(\omega',\Omega)} \lesssim \| \Psi \|_{L^{p'}(\omega',\Omega)}.
\end{equation}
% Assume the claim is proved.
With the estimate \eqref{eq:LqAuxEstimate} at hand, we invoke  the G\aa{}rding--like inequality \eqref{eq:first_weighted_bound} of Lemma~\ref{lem:Garding} with $\bmu = \boldsymbol0$ and $\gamma = \Psi$ to obtain
\begin{equation}\label{eq:LqAuxEstimate-consequence}
  \| \GRAD w \|_{\bL^{p'}(\omega',\Omega)} \lesssim \| \Psi \|_{L^{p'}(\omega',\Omega)} + \| w \|_{L^{p'}(\omega',\Omega)} \lesssim \| \Psi \|_{L^{p'}(\omega',\Omega)},
\end{equation}
where we used \eqref{eq:LqAuxEstimate} in the last step.

Let now $\Phi \in W^{-1,p}(\omega,\Omega)$ be smooth and given. Consider $v$ as the solution to $\calL v = \Phi$ in $\Omega$ together with the Dirichlet boundary conditions $v = 0$ on $\partial \Omega$. In view of \cite[Theorem 3.2.1.2]{MR3396210} and \cite[Theorem 1]{MR1198129}, we have that $v \in W^{1,2}_0(\Omega) \cap W^{2,2}(\Omega) \cap W^{1,q}_0(\Omega)$ for every $q \in (1,\infty)$. Again, we can choose $q$ sufficiently large so that $v \in W^{1,q}_0(\Omega) \subset W^{1,p}(\omega,\Omega)$. We now compute
% As a consequence, if for a given  $\Phi \in W^{-1,p}(\omega,\Omega)$ but smooth,  $ v \in W^{1,2}_0(\Omega)$ solves  $\calL v = \Phi,$ then
% $v\in W^{2,2}(\Omega) \cap W^{1,2}_0(\Omega) \cap C^{0,1}(\bar\Omega) \subset W^{1,p}(\omega,\Omega)$, and therefore,
% we have that
%Moreover, we have tha following computations valid:
\begin{align*}
\| v \|_{L^p(\omega,\Omega)}
    & = \sup_{\Psi \in C_0^\infty(\Omega)} \frac{\int_\Omega v \Psi \diff x}{\| \Psi \|_{L^{p'}(\omega',\Omega)}}
    =
    \sup_{\Psi \in C_0^\infty(\Omega)} \frac{\int_\Omega v \calL^* w \diff x}{\| \Psi \|_{L^{p'}(\omega',\Omega)}} \\
    &=
    \sup_{\Psi \in C_0^\infty(\Omega)} \frac{\langle \Phi , w \rangle \diff x}{\| \Psi \|_{L^{p'}(\omega',\Omega)}}
    \leq
    \|\Phi \|_{W^{-1,p}(\omega,\Omega)} \sup_{\Psi \in C_0^\infty(\Omega)} \frac{\| \GRAD w \|_{\bL^{p'}(\omega',\Omega)}}{\| \Psi \|_{L^{p'}(\omega',\Omega)}}\\
    &
    \lesssim \|\Phi \|_{W^{-1,p}(\omega,\Omega)},
\end{align*}
where we used \eqref{eq:LqAuxEstimate-consequence} in the last step. To summarize, if \eqref{eq:LqAuxEstimate} holds and $v \in W^{1,p}_0(\omega,\Omega)$ solves  $\calL v = \Phi$ in $\Omega$ and $v = 0$ on $\partial \Omega$, then
\[
  \| v \|_{L^p(\omega,\Omega)} \lesssim \|\Phi \|_{W^{-1,p}(\omega,\Omega)}.
\]
Since $\bmu$ and $\gamma$ are smooth by assumption, $\Phi = -\DIV \bmu + \gamma$ is also smooth and belongs to $W^{-1,p}(\omega,\Omega)$. We now apply the above estimate for $v = \fraku$ and $\Phi = -\DIV \bmu + \gamma$ and combine it with Lemma~\ref{lem:Garding} to conclude \eqref{eq:aprioriForSmooth}.

It therefore remains to prove \eqref{eq:LqAuxEstimate}. To do so, we can use, for instance, the representation of $w$ via the Green's function $\calG$ of $\calL^*$. In fact, we have
\[
  w(x) = \int_\Omega \calG(x,y) \Psi(y) \diff y.
\]
Since $\Omega$ is convex, the following bounds follow from \cite[Theorem 4.1, estimate (4.4)]{MR3169756} and \cite[Theorem 3.3, item (i)]{MR657523}: For $x,y \in \Omega$, we have that
\[
  |\calG(x,y)| \leq K \begin{dcases}
                    1+ |\log(|x-y|)|, & d=2, \\
                    |x-y|^{2-d}, & d = 3.
                  \end{dcases}
\]
In these estimates, the constant $K$ depends on $\bB$ only through $\alpha$ and $\|\bB \|_{\calbC(\bar\Omega)}$. The estimate \eqref{eq:LqAuxEstimate} then follows quickly from the continuity in $L^{r}(\omega,\Omega)$, with $r \in (1,\infty)$ and $\omega \in A_r$, of integral operators with standard kernels, i.e., kernels that satisfy the estimates (5.12)--(5.14) in \cite[page 99]{MR1800316}. Since $\Omega$ is bounded, it follows that for $d=2$
% A very suboptimal way of seeing this is by invoking that $\Omega$ is bounded. Then,
\begin{align*}
  |w(x)| &\leq K \int_\Omega \left( 1+ |\log(|x-y|)|\right) |x-y|^2 \frac{|\Psi(y)|}{|x-y|^2} \diff y   \\
  &\leq K \diam(\Omega)^2 \left( 1 +  |\log(\diam(\Omega))|\right) \int_\Omega \frac{|\Psi(y)|}{|x-y|^2} \diff y,
\end{align*}
where we used that $\mathbb{R}^{+} \ni t \mapsto t^2 |\log(t)| \in \mathbb{R}$ is continous and $t^2 |\log(t)| \rightarrow 0$ as $t \downarrow 0$. For $d = 3$, we have that
\[
  |w(x)| \leq K \diam(\Omega)^2 \int_\Omega \frac{|\Psi(y)|}{|x-y|^3} \diff y.
\]
Define $\mathfrak{K}: \mathbb{R}^d \times \mathbb{R}^d \setminus \{(x,x) : x \in \mathbb{R}^d \} \rightarrow \mathbb{R}$ as $\mathfrak{K}(x,y) = |x-y|^{-d}$. It follows that $\mathfrak{K}$ is a standard kernel in the sense that it satisfies the estimates (5.12)--(5.14) in \cite[page 99]{MR1800316}. We are thus in a position to apply the continuity of singular integral operators on Muckenhoupt weighted spaces \cite[Theorem 7.11, page 144]{MR1800316}, which guarantees the bound \eqref{eq:LqAuxEstimate}: $\| w \|_{L^{p'}(\omega',\Omega)} \lesssim \| \Psi \|_{L^{p'}(\omega',\Omega)}$. This concludes the proof.
\end{proof}

We are now ready to present a proof for the main result of this section.

\begin{proof}[Proof of Theorem~\ref{thm:LinearWellPosed}]
We begin with the proof of existence by means of an approximation argument. Since $\Omega$ is bounded and convex and $\bA \in \calbC_\sym(\bar\Omega)$, there is a sequence $\{\bA_k \}_{k=1}^\infty \subset \calbC_\sym^\infty(\bar\Omega)$ such that $\bA_k \rightrightarrows \bA$ in $\calbC_\sym(\bar\Omega)$. Due to the  uniform convergence, it can be assumed without loss of generality that for each $k$ the matrix valued function $\bA_k$ satisfies the same lower spectral bound as $\bA$, \ie the third condition in \eqref{eq:Aislinear}. On the other hand, since $\omega \in A_p$, we can find $\{\bef_k\}_{k=1}^\infty \subset \bC^\infty(\Omega)$ such that $\bef_k \to \bef$ in $\bL^p(\omega,\Omega)$ and $\{g_k\}_{k=1}^\infty \subset C^\infty(\Omega)$ such that $g_k \to g$ in $L^p(\omega,\Omega)$.

We now consider, for each $k \in \mathbb{N}$, the problem:
\[
  -\DIV (\bA_k\GRAD u_k) = -\DIV \bef_k + g_k \text{ in } \Omega, \qquad u_k = 0 \text{ on } \partial \Omega.
\]
The smoothness of the data is more than enough to confirm the existence and uniqueness of a solution $u_k \in W^{1,2}_0(\Omega) \cap W^{2,2}(\Omega)$ for each $k \in \mathbb{N}$. Moreover, since $\Omega$ is convex, we have that $u_k \in W^{1,q}_0(\Omega)$ for all $q \in (1,\infty)$ \cite[Theorem 1]{MR1198129}. In particular, due to the reverse H\"older inequality, $u_k \in W^{1,p}_0(\omega,\Omega)$. We can therefore use \eqref{eq:aprioriForSmooth} to obtain the following bound for each $k \in \mathbb{N}$:
\[
  \| \GRAD u_k \|_{\bL^p(\omega,\Omega)} \leq C \left( \| \bef_k \|_{\bL^p(\omega,\Omega)} + \| g_k \|_{L^p(\omega,\Omega)} \right),
\]
where $C$ is uniform in $k$ and depends only on $\alpha$ and $\| \bA \|_{\calbC(\bar\Omega)}$. We may then pass to the limit $k \uparrow \infty$ and obtain, up to a subsequence,
\[
 u \in W^{1,p}_0(\omega,\Omega): u_k \rightharpoonup u
 \textrm{ in }
 W^{1,p}_0(\omega,\Omega)
 \textrm{ as } k \uparrow \infty.
\]
%
% $u \in W^{1,p}_0(\omega,\Omega)$ such that $u_k \rightharpoonup u$ in $W^{1,p}_0(\omega,\Omega)$.
Linearity asserts that $u$ is a solution to \eqref{eq:Diening}. Since $u$ is a solution and belongs to $W^{1,p}_0(\omega,\Omega)$, we can apply the estimate \eqref{eq:aprioriForSmooth} of Corollary \ref{co:aprioriForSmooth} to obtain \eqref{eq:stabilityLinear}.

It therefore remains to show uniqueness. Due to linearity, it suffices to show that the homogeneous problem has only the trivial solution, \ie $\bef = \boldsymbol0$ and $g = 0$ imply that $u=0$. Let $u \in W^{1,p}_0(\omega,\Omega)$ be a solution of the homogeneous problem. As a consequence of the reverse H\"older inequality, there exists $\vare >0$ such that $u$ belongs to $W^{1,1+\vare}_0(\Omega)$. Since $\bA \in \calbC_\sym(\bar\Omega)$, we can conclude from \cite[Theorem A5.1]{MR2548032} that $u \equiv 0$ in $\Omega$. This argument shows uniqueness of solutions and concludes the proof.
\end{proof}

\section{Quasilinear problems}
\label{sec:Quasilinear}

Having obtained well posedness for linear problems, we prove the existence and uniqueness of solutions to \eqref{eq:Diening}. This extends the results of \cite{MR3531368} to the case of convex polytopal domains.

\begin{theorem}[existence and uniqueness]
\label{thm:DieningIsWellPosed}
Let $d \in \{2,3\}$ and let $\Omega \subset \R^d$ be a bounded and convex polytope. Let us assume that $\ba$ satisfies the structural assumptions \eqref{item:A}--\eqref{item:F}. Let $p \in (1,\infty)$ and let $\omega \in A_p$. Then, for any $\bef \in \bL^p(\omega,\Omega)$ and $g \in L^p(\omega,\Omega)$, there exists a unique solution $u \in W^{1,p}_0(\omega,\Omega)$ to problem \eqref{eq:Diening} in the sense that
\begin{equation}
\label{eq:ToLinearize}
  \int_\Omega \ba(x,\GRAD u) \cdot \GRAD \varphi \diff x = \int_\Omega \bef \cdot \GRAD \varphi \diff x + \int_\Omega g \varphi \diff x \qquad \forall \varphi \in C_0^\infty(\Omega).
\end{equation}
Moreover, the following a priori estimate holds
\begin{equation}
  \| \GRAD u \|_{\bL^p(\omega,\Omega)} \lesssim 1+ \| \bef \|_{\bL^p(\omega,\Omega)} + \| g \|_{L^p(\omega,\Omega)}.
  \label{eq:aprioriWeighted}
\end{equation}
The hidden constant is independent of $\bef$, $g$, and $u$; it depends on $\omega$ only through
% the Muckenhoupt characteristic
$[\omega]_{A_p}$.
\end{theorem}
\begin{proof}
We follow, with some modifications, the proof of \cite[Theorem 2.3]{MR3531368}.
% Let us then sketch the result.
We begin the proof by showing that the estimate \eqref{eq:aprioriWeighted} holds for an arbitrary $u \in W_{0}^{1,q}(\Omega)$ solving \eqref{eq:ToLinearize} ($q>1$). Without loss of generality we can restrict ourselves to the case $q \in (1,2)$. We now note that due to the \emph{Rubio de Francia
extrapolation} (see \cite{RubioDeFrancia,MR0745140,MR0684631}, and \cite[Theorem 7.8]{MR1800316}), it suffices to prove the desired estimate for the case $p=2$. For this purpose, we define the following weight for $j \in \polN$:
\[
  \omega_j = \min\left\{ \omega, j \left( 1 + \calM[\GRAD u] \right)^{q-2} \right\}.
\]
Since $q \in (1,2)$, it follows from the item \eqref{itemvi} in \S \ref{sec:Notation} that $\calM[\GRAD u]^{q-2} \in A_2$. We now use that $\omega \in A_2$, \cite[bound (3.6)]{MR3531368}, and the fact that the $A_p$ condition is invariant under translations and dilations \cite[Remark 1.2.4, item 5]{MR1774162} to obtain that $\omega_j \in A_2$ and that
\[
 [\omega_j]_{A_2} \leq  [\omega]_{A_2} + [( 1 + \calM[\GRAD u])^{q-2} ]_{A_2} \leq C(u,\omega).
\]
% $[\omega_j]_{A_2} \leq  [\omega]_{A_2} + [( 1 + \calM[\GRAD u])^{q-2} ]_{A_2} \leq C$, where $C$ depends on $u$ and $\omega$.
By construction, we have that $u \in W^{1,2}_0(\omega_j,\Omega)$, $g \in L^2(\omega_j,\Omega)$, and $\bef \in \bL^2(\omega_j,\Omega)$.

The reformulation of problem \eqref{eq:ToLinearize} results in the following relation
\[
  \int_\Omega \bA(x) \GRAD u \cdot \GRAD \varphi \diff x = \int_\Omega \bef \cdot\GRAD \varphi \diff x + \int_\Omega g \varphi \diff x + \int_\Omega \left[ \bA(x) \GRAD u - \ba(x,\GRAD u) \right] \cdot \GRAD \varphi \diff x,
\]
for all $\varphi \in C_0^\infty(\Omega)$. Let us introduce $\calF \in W^{-1,2}(\omega_j,\Omega)$ as
\[
  \langle \calF, \varphi \rangle = \int_\Omega \bef \cdot\GRAD \varphi \diff x + \int_\Omega g \varphi \diff x + \int_\Omega \left[ \bA(x) \GRAD u - \ba(x,\GRAD u) \right] \cdot \GRAD \varphi \diff x.
\]
We then use the bound \eqref{eq:stabilityLinear} from Theorem \ref{thm:LinearWellPosed} in order to obtain
\begin{multline}
\int_{\Omega} \omega_j | \GRAD u |^2 \diff x
\lesssim
\int_{\Omega} \omega_j | \bef |^2 \diff x
+
\int_{\Omega} \omega_j | g |^2 \diff x
+
\int_{\Omega} \omega_j | \bA(x) \GRAD u - \ba(x,\GRAD u) |^2 \diff x.
\label{eq:auxiliary_estimate_omegaj_nablau}
\end{multline}
The hidden constant depends on the weight $\omega$ only through ts Muckenhoupt characteristic $[\omega]_{A_p}$ and on the coefficient $\bA$ only through the norm $\|\bA\|_{\calbC(\bar\Omega)}$. This step is the main difference from \cite{MR3531368}, where the authors instead use \cite[Theorem 2.5]{MR3531368} to derive the previous bound. The results of \cite[Theorem 2.5]{MR3531368} hold under the \emph{additional assumption that $\Omega$ is a $C^1$-domain.}
% , and see the estimate for the dual norm below.  With this at hand,
% %have reached a point where we depart from \cite{MR3531368}. I
% instead of invoking \cite[Theorem 2.5]{MR3531368} as it was done in  \cite{MR3531368}, we use Theorem~\ref{thm:LinearWellPosed} to conclude that
% \begin{equation}
% \label{est-of-u:nonlinear}
%   \| \GRAD u \|_{\bL^2(\omega_j,\Omega)} \lesssim \| \calF \|_{W^{-1,2}(\omega_j,\Omega)},
% \end{equation}
% where the implicit constant depends on $\omega_j$ only through its $A_2$ characteristic and on $\bA$ only through $\| \bA \|_{\calbC(\bar\Omega)}$. Let us now estimate the dual norm $\| \calF \|_{W^{-1,2}(\omega_j,\Omega)}.$ Clearly,
% \[
%   \| \calF \|_{W^{-1,2}(\omega_j,\Omega)} \leq \| \bef \|_{\bL^2(\omega_j,\Omega)} + \| g \|_{L^2(\omega_j,\Omega)} + \sup_{\varphi \in C_0^\infty(\Omega)} \frac1{\|\GRAD \varphi \|_{\bL^2(\omega_j^{-1},\Omega)} } \calA(\varphi),
% \]
% with
% \begin{align*}
%   \calA(\varphi) &= \left|\int_\Omega \left[ \bA(x) \GRAD u - \ba(x,\GRAD u) \right] \cdot \GRAD \varphi \diff x \right|\\
%   & \leq \left( \int_\Omega \omega_j |\bA(x) \GRAD u - \ba(x,\GRAD u) |^2 \diff x \right)^{1/2} \|\GRAD \varphi \|_{\bL^2(\omega_j^{-1},\Omega)}.
% \end{align*}
We now let $N>0$ to be chosen and define the sets $S_N = \{ x \in \Omega: |\GRAD u(x)| \leq N \}$ and $B_N = \Omega \setminus S_N$, and observe that
\begin{align*}
  \int_{S_N} \omega_j |\bA(x) \GRAD u - \ba(x,\GRAD u) |^2 \diff x &\leq 2 \int_{S_N} \omega_j \left( |\bA(x)|^2|\GRAD u |^2 + |\ba(x,\GRAD u)|^2 \right) \diff x \\
  &
  \leq
  2N^2 \left( \|\bA\|^2_{\calbC(\bar\Omega)}
+ \Lambda^2 \right) \int_\Omega \omega_j \diff x,
\end{align*}
where we have used \eqref{eq:baisgrowth} and the fact that $\bA \in \calbC_{\sym}(\bar\Omega)$. Let now $\epsilon >0$. We can choose $N=N(\epsilon)>0$ so large that we can claim with \eqref{eq:baisUhlenbeck} that
\begin{align*}
  \int_{B_N} \omega_j |\bA(x) \GRAD u - \ba(x,\GRAD u) |^2 \diff x &= \int_{B_N} \frac{ |\bA(x) \GRAD u - \ba(x,\GRAD u) |^2}{|\GRAD u|^2} \omega_j |\GRAD u|^2 \diff x \\
  &\leq \epsilon \int_\Omega \omega_j|\GRAD u|^2  \diff x.
\end{align*}
If we combine the previous two estimates and insert them into \eqref{eq:auxiliary_estimate_omegaj_nablau}, we get
\[
\int_{\Omega} \omega_j | \GRAD u |^2 \diff x
\lesssim
\int_{\Omega} \omega_j | \bef |^2 \diff x
+
\int_{\Omega} \omega_j | g |^2 \diff x
+
N^2 \int_{\Omega} \omega_{j} dx
+
\epsilon \int_{\Omega} \omega_j | \GRAD u |^2 \diff x.
\]
Note that, once $\epsilon>0$ is chosen, $N$ is fixed independently of $j$ and the hidden constant in the previous bound is independent of $j$. We can then choose $\epsilon$ small so that $C \epsilon \int_{\Omega} \omega_j | \GRAD u |^2 \diff x$ on the right hand side of this estimate can be absorbed into the left hand side and we obtain
\[
\int_{\Omega} \omega_j | \GRAD u |^2 \diff x \lesssim \int_\Omega \omega_j \left( 1 + |\bef|^2 + |g|^2 \right) \diff x.
\]

The rest of the proof follows the arguments in \cite[Theorem 2.3]{MR3531368} without change. This concludes the proof.
\end{proof}

\begin{remark}
The generalization of \cite{MR3531368} to incompressible fluids can be found in \cite{MR3582412}. In this paper, the authors consider a model of a steady incompressible non-Newtonian flow under external forcing and provide the full-range theory, i.e., existence, optimal regularity, and uniqueness of solutions with respect to forcing belonging to Lebesgue spaces and Muckenhoupt weighted Lebesgue spaces. For extensions of some results from \cite{MR3582412} to convex polyhedral domains we refer the reader to \cite{MR4408483}.
\end{remark}

\section{Discretization}
\label{sec:Discretize}

In this section, we propose and analyze a convergent finite element discretization for problem \eqref{eq:Diening}. We begin the discussion by noting that since we assume that $\Omega$ is a convex polytope, it can be meshed exactly. Therefore, we introduce a quasiuniform family $\Tr = \{ \T \}_{h>0}$ of meshes $\T$ of $\bar{\Omega}$, where the parameter $h$ is the mesh size. Next, we define the family of finite element spaces $\{\V\}_{h>0}$ as
\[
  \V = \left\{ w_h \in W^{1,\infty}_0(\Omega) : \ w_{h|T} \in \P \,\, \forall T \in \T \right\},
\]
where, for a simplicial element $T$, $\P$ corresponds to $\mathbb{P}_1$---the space of polynomials of total degree at most one. If $T$ is a $d$--rectangle, then $\P$ stands for $\polQ_1$---the space of polynomials of degree not larger than $1$ in each variable.
% every $T \in \T$ is equivalent to a simplex, whereas $\P = \polQ_1$ (polynomials of degree at most one in each variable) if $T$ is equivalent to a cube.
We immediately note that, for every $p \in (1,\infty)$ and all $\omega \in A_p$, we have
\[
  \V \subset W^{1,p}_0(\omega,\Omega) \cap W^{1,p'}_0(\omega',\Omega).
\]

\subsection{The Ritz projection}
We define $\Ritz : W^{1,1}_0(\Omega) \to \V$ as
\begin{equation}
\label{eq:DefOfRitz}
  \int_\Omega \GRAD (w-\Ritz w) \cdot \GRAD \varphi_h \diff x = 0 \quad \forall \varphi_h \in \V.
\end{equation}
The Ritz projection $\Ritz$ is the best approximation to a function with respect to the $W^{1,2}_0(\Omega)$--norm. Moreover, $\Ritz$ is stable and has optimal approximation properties in weighted spaces.

\begin{theorem}[stability, approximation, and inf-sup condition]
\label{thm:StabRitz}
Let $\Tr = \{ \T \}_{h>0}$ be a quasiuniform family of triangulations of the convex polytope $\Omega$. For every $p \in (1,\infty)$ and all $\omega \in A_p$ the Ritz projection $\Ritz : W^{1,1}_0(\Omega) \to \V$ defined in \eqref{eq:DefOfRitz} satisfies
 \[
  \| \GRAD \Ritz w \|_{\bL^p(\omega,\Omega)} \leq C_R \| \GRAD w \|_{\bL^p(\omega,\Omega)} \qquad \forall w \in W^{1,p}_0(\omega,\Omega).
\]
where $C_R$ is independent of $w$ and $h$. The mapping $\Ritz$ also has the optimal approximation property
\[
  \| \GRAD (w-\Ritz w) \|_{\bL^p(\omega,\Omega)} \leq (1+C_R) \inf_{\varphi_h \in \V} \| \GRAD ( w - \varphi_h) \|_{\bL^p(\omega,\Omega)} \quad \forall w \in W^{1,p}_0(\omega,\Omega).
\]
Finally, we have the following discrete inf-sup condition:
\[
  \| \GRAD w_h \|_{\bL^p(\omega,\Omega)} \leq C_\Delta C_R \sup_{\varphi_h \in \V} \frac{ \int_\Omega \GRAD w_h \cdot \GRAD \varphi_h \diff x}{\| \GRAD \varphi_h \|_{\bL^{p'}(\omega',\Omega)}} \qquad \forall w_h \in \V.
\]
\end{theorem}
\begin{proof}
The stability is proved in  \cite[Corollary 3.6]{diening2023pointwise}. The best approximation follows directly from the fact that $\Ritz$ is a projection. Finally, the claimed discrete inf-sup condition is obtained as follows: From Proposition~\ref{prop:weightedLaplace}, the definition of $\Ritz$, the stability of $\Ritz$, and a trivial fact it follows that
\begin{align*}
  \| \GRAD w_h \|_{\bL^p(\omega,\Omega)} &\leq C_\Delta\sup_{\varphi \in W^{1,p'}_0(\omega',\Omega)} \frac{ \int_\Omega \GRAD w_h \cdot \GRAD \varphi \diff x}{\| \GRAD \varphi \|_{\bL^{p'}(\omega',\Omega)}} \\
    &= C_\Delta\sup_{\varphi \in W^{1,p'}_0(\omega',\Omega)} \frac{ \int_\Omega \GRAD w_h \cdot \GRAD \Ritz \varphi \diff x}{\| \GRAD \varphi \|_{\bL^{p'}(\omega',\Omega)}} \\
    &\leq C_\Delta C_R \sup_{\varphi \in W^{1,p'}_0(\omega',\Omega)} \frac{ \int_\Omega \GRAD w_h \cdot \GRAD \Ritz \varphi \diff x}{\| \GRAD \Ritz \varphi \|_{\bL^{p'}(\omega',\Omega)}} \\
    &\leq C_\Delta C_R \sup_{\varphi_h \in \V} \frac{ \int_\Omega \GRAD w_h \cdot \GRAD \varphi_h \diff x}{\| \GRAD \varphi_h \|_{\bL^{p'}(\omega',\Omega)}}.
\end{align*}
This concludes the proof.
\end{proof}

\subsection{The linear problem}
\label{sub:LinearForFEM}

Having established from the properties of the Ritz projection that the Poisson problem, i.e., \eqref{eq:Diening} in the framework of \eqref{eq:Aislinear} with $\bA = \bI$, is discretely well posed, let us examine the case of a variable coefficient. While we conjecture that the results we present here remain true in the case of a merely continuous and elliptic $\bA$, we must make the following assumptions:
\begin{equation}
\label{eq:AsmallOsc}
  \bA \in \calbC_{\sym}(\bar\Omega), \qquad \alpha \bI \preceq \bA(x) \preceq \Lambda \bI \ \ \forall x \in \bar\Omega,
  \qquad 2C_\Delta C_R \left( 1- \frac\alpha\Lambda \right)\leq 1.
\end{equation}
With the latter assumption of ``\emph{small oscillation}'', we follow \cite[Proposition 8.6.2]{MR2373954} and obtain the following discrete stability result.

\begin{theorem}[discrete inf-sup condition]
\label{thm:DLinearWellPosed}
Let $p \in (1,\infty)$, and let $\omega \in A_p$. Let us assume that the coefficient $\bA$ satisfies \eqref{eq:AsmallOsc}. Then, we have
\begin{equation}
\label{eq:DInfSupWithA}
  \| \GRAD w_h \|_{\bL^p(\omega,\Omega)} \leq \frac{2C_\Delta C_R}\Lambda \sup_{\varphi_h \in \V } \frac{ \int_\Omega \bA(x) \GRAD w_h \cdot \GRAD \varphi_h \diff x}{\| \GRAD \varphi_h \|_{\bL^{p'}(\omega',\Omega)}} \quad \forall w_h \in \V.
\end{equation}
In particular, if $\bef \in \bL^p(\omega,\Omega)$ and $g \in L^p(\omega,\Omega)$, then there is a unique solution $u_h$ of the following discrete problem: Find $u_h \in \V$ such that
\[
  \int_\Omega \bA(x) \GRAD u_h \cdot \GRAD v_h \diff x= \int_\Omega \bef \cdot \GRAD v_h \diff x + \int_\Omega g v_h \diff x \qquad \forall v_h \in \V.
\]
In addition, the solution $u_h$ satisfies the estimate
\[
  \| \GRAD u_h \|_{\bL^p(\omega,\Omega)} \leq \frac{2C_\Delta C_R}\Lambda \left( \| \bef \|_{\bL^p(\omega,\Omega)} + C_P \| g \|_{L^p(\omega,\Omega)} \right).
\]
$C_P$ is the best constant in $\| w \|_{\bL^{p}(\omega,\Omega)} \leq C_P \| \GRAD w \|_{\bL^{p}(\omega,\Omega)}$ for all $w \in W_0^{1,p}(\omega,\Omega)$.
\end{theorem}
\begin{proof}
Let us begin with the proof of \eqref{eq:DInfSupWithA}. On the basis of Theorem~\ref{thm:StabRitz} we have
\[
  \| \GRAD w_h \|_{\bL^p(\omega,\Omega)} \leq C_\Delta C_R  \sup_{\varphi_h \in \V } \frac{ \int_\Omega \GRAD w_h \cdot \GRAD \varphi_h \diff x}{\| \GRAD \varphi_h \|_{\bL^{p'}(\omega',\Omega)}} \quad \forall w_h \in \V.
\]
From the preceding bound it follows that the following estimate is valid
\begin{multline}
  \| \GRAD w_h \|_{\bL^p(\omega,\Omega)} \leq \frac{C_\Delta C_R}\Lambda \sup_{\varphi_h \in \V } \frac{ \int_\Omega \bA(x) \GRAD w_h \cdot \GRAD \varphi_h \diff x}{\| \GRAD \varphi_h \|_{\bL^{p'}(\omega',\Omega)}} \\ + C_\Delta C_R \sup_{\varphi_h \in \V } \frac{ \int_\Omega (\bI - \tfrac1\Lambda \bA(x) ){\GRAD w_h} \cdot \GRAD \varphi_h \diff x}{\| \GRAD \varphi_h \|_{\bL^{p'}(\omega',\Omega)}} = \mathrm{I} + \mathrm{II}.
  \label{eq:I+II}
\end{multline}
To control the term $\mathrm{II}$, we note that from the spectral inequalities $\alpha \bI \preceq \bA(x) \preceq \Lambda \bI$, which hold for each $x \in \bar{\Omega}$, it follows that
\[
  \left( \frac\alpha\Lambda - 1 \right) \bI \preceq \frac1\Lambda \bA(x) - \bI \preceq \bO, \qquad \forall x \in \Omega,
\]
where $\bO$ is the zero matrix. Consequently,
\[
  \left\| \frac1\Lambda \bA - \bI \right\|_{\calbC(\bar\Omega)}
  = \sup_{x \in \bar{\Omega}} \left| \frac1\Lambda \bA(x) - \bI  \right|
  \leq 1- \frac\alpha\Lambda.
\]
On the basis of this estimate, we can conclude that the term labeled as $\mathrm{II}$ can be bounded as follows:
\[
  \mathrm{II} \leq C_\Delta C_R \left\| \frac1\Lambda \bA - \bI \right\|_{\calbC(\bar\Omega)} \| \GRAD w_h \|_{\bL^p(\omega,\Omega)}
   \leq \frac12 \| \GRAD w_h \|_{\bL^p(\omega,\Omega)},
\]
where we have used \eqref{eq:AsmallOsc}. Replacing this estimate into \eqref{eq:I+II} yields \eqref{eq:DInfSupWithA}.
% Gathering all obtained bounds \eqref{eq:DInfSupWithA} follows.

Finally, we see that the existence and uniqueness of $u_h$ is immediate, as is a dimension--dependent (read $h$) estimate.
% well as some dimension (read $h$) dependent estimate.
The issue is to obtain a uniform estimate in the correct norm. Such an estimate can be obtained as a consequence of \eqref{eq:DInfSupWithA}.
\end{proof}

\begin{remark}[small oscillation]
The small oscillation condition \eqref{eq:AsmallOsc} essentially means that the spread of the eigenvalues of the matrix $\bA$ is globally bounded. As mentioned above, we conjecture that condition \eqref{eq:AsmallOsc} is merely an artifact of our proof, which is of a perturbative nature.
\end{remark}

\subsection{The quasilinear problem}

Finally, we consider a convergent discretization of \eqref{eq:Diening}. Namely, for $p \in (1,\infty)$, $\omega \in A_p$, $\bef \in \bL^p(\omega,\Omega)$, and $g \in L^p(\omega,\Omega)$ we seek for $u_h \in \V$ such that 
\begin{equation}
\label{eq:BIdiscrete}
  \int_\Omega \ba(x,\GRAD u_h) \cdot \GRAD v_h \diff x = \int_\Omega \bef \cdot \GRAD v_h \diff x + \int_\Omega g v_h \diff x \quad \forall v_h \in \V.
\end{equation}
Since this is a finite dimensional problem, the coercivity \eqref{eq:baiscoercive}, the strict monotonicity \eqref{eq:baismonotone}, and the growth \eqref{eq:baisgrowth} assumptions on $\ba$ guarantee the existence and uniqueness of a solution $u_h \in \V$. The important point here is again the uniform estimates and the convergence. This is the content of the next result.

\begin{theorem}[uniform estimates and convergence]
Assume that $\ba$ satisfies the structural assumptions \eqref{item:A}--\eqref{item:F}. Let us also assume that $\bA$ satisfies the small oscillation condition \eqref{eq:AsmallOsc}.  Let $u_h \in \V$ be the solution to \eqref{eq:BIdiscrete}. Then,
\[
  \| \GRAD u_h \|_{\bL^p(\omega,\Omega)} \lesssim 1 +  \| \bef \|_{\bL^p(\omega,\Omega)} + \| g \|_{L^p(\omega,\Omega)}.
\]
Moreover, there exists a nonrelabeled subsequence $\{ u_h \}_{h>0} \subset \V$ such that $u_h \rightharpoonup u$ in $W^{1,p}_0(\omega,\Omega)$ as $h \rightarrow 0$. More importantly, $u$ solves \eqref{eq:Diening} in the sense that \eqref{eq:ToLinearize} holds.
\end{theorem}
\begin{proof}
To obtain the uniform bound, we mimic the proof of Theorem~\ref{thm:DieningIsWellPosed}. In fact, we rewrite the discrete equation \eqref{eq:BIdiscrete} as
\[
  \int_\Omega \bA(x) \GRAD u_h \cdot \GRAD v_h \diff x = \int_\Omega \left[\bef \cdot \GRAD v_h + g v_h \right] \diff x + \int_\Omega \left[ \bA(x) \GRAD u_h - \ba(x,\GRAD u_h ) \right] \cdot \GRAD v_h \diff x
\]
for all $v_h \in \V$, and invoke Theorem~\ref{thm:DLinearWellPosed} to assert that
\[
  \| \GRAD u_h \|_{\bL^p(\omega,\Omega)} \leq \frac{2C_\Delta C_R}\Lambda \left( \| \bef \|_{\bL^p(\omega,\Omega)} + C_P \| g \|_{L^p(\omega,\Omega)} + \| \calF_h \|_{W^{-1,p}(\omega,\Omega)} \right),
\]
where $\calF_h: W^{1,p'}_0(\omega',\Omega) \rightarrow \mathbb{R}$ is defined by
\[
  \langle \calF_h, v \rangle = \int_\Omega \left[ \bA(x) \GRAD u_h - \ba(x,\GRAD u_h ) \right] \cdot \GRAD v \diff x.
%   \quad \forall v \in W^{1,p'}_0(\omega',\Omega).
\]
The rest of the proof of the estimate follows \emph{verbatim} the proof of Theorem~\ref{thm:DieningIsWellPosed}.

Since the family $\{u_h \}_{h>0}$ is uniformly bounded in $W^{1,p}_0(\omega,\Omega)$, we can extract a weakly convergent subsequence. The proof that this limit is indeed a solution to \eqref{eq:Diening} follows from \cite[Section 4B]{MR3531368}, where the existence is shown by means of an approximation argument; see also \cite[Theorem 15]{MR4408483}. This concludes the proof.
\end{proof}

\bibliographystyle{amsplain}
\bibliography{biblio}

\providecommand{\bysame}{\leavevmode\hbox to3em{\hrulefill}\thinspace}
\providecommand{\MR}{\relax\ifhmode\unskip\space\fi MR }
\providecommand{\MRhref}[2]{%
  \href{http://www.ams.org/mathscinet-getitem?mr=#1}{#2}
}
\providecommand{\href}[2]{#2}
\begin{thebibliography}{10}

\bibitem{Adimurthi2021}
K.~Adimurthi, T.~Mengesha, and N.~C. Phuc, \emph{Gradient weighted norm
  inequalities for linear elliptic equations with discontinuous coefficients},
  Applied Mathematics \& Optimization \textbf{83} (2021), no.~1, 327--371.

\bibitem{MR1198129}
Yu.~A. Alkhutov and V.~A. Kondrat'ev, \emph{Solvability of the {D}irichlet
  problem for second-order elliptic equations in a convex domain},
  Differentsial'nye Uravneniya \textbf{28} (1992), no.~5, 806--818, 917.
  \MR{1198129}

\bibitem{MR2548032}
A.~Ancona, \emph{Elliptic operators, conormal derivatives and positive parts of
  functions}, J. Funct. Anal. \textbf{257} (2009), no.~7, 2124--2158, With an
  appendix by Ha\"{\i}m Brezis. \MR{2548032}

\bibitem{MR2373954}
S.~C. Brenner and L.~R. Scott, \emph{The mathematical theory of finite element
  methods}, third ed., Texts in Applied Mathematics, vol.~15, Springer, New
  York, 2008. \MR{2373954}

\bibitem{MR3531368}
M.~Bul\'\i{\v c}ek, L.~Diening, and S.~Schwarzacher, \emph{Existence,
  uniqueness and optimal regularity results for very weak solutions to
  nonlinear elliptic systems}, Anal. PDE \textbf{9} (2016), no.~5, 1115--1151.
  \MR{3531368}

\bibitem{MR3582412}
M.~Bul\'{\i}\v{c}ek, J.~Burczak, and S.~Schwarzacher, \emph{A unified theory
  for some non-{N}ewtonian fluids under singular forcing}, SIAM J. Math. Anal.
  \textbf{48} (2016), no.~6, 4241--4267. \MR{3582412}

\bibitem{byun2023existence}
S-S. Byun and M.~Lim, \emph{Existence of very weak solutions to nonlinear
  elliptic equation with nonstandard growth and global weighted gradient
  estimates}, arXiv:2311.11479, 2023.

\bibitem{BF03377367}
S-S. Byun and S.~Ryu, \emph{Weighed estimates for nonlinear elliptic problems
  with {O}rlicz data}, Journal of Elliptic and Parabolic Equations \textbf{1}
  (2015), no.~1, 49--61.

\bibitem{RubioDeFrancia}
D.~Cruz-Uribe, J.M. Martell, and C.~P{\'e}rez, \emph{Weights, extrapolation and
  the theory of {R}ubio de {F}rancia}, Operator Theory: Advances and
  Applications, vol. 215, Springer Basel AG, 2011.

\bibitem{MR2790542}
L.~Diening, P.~Harjulehto, P.~H\"{a}st\"{o}, and M.~Ru\v{z}i\v{c}ka,
  \emph{Lebesgue and {S}obolev spaces with variable exponents}, Lecture Notes
  in Mathematics, vol. 2017, Springer, Heidelberg, 2011. \MR{2790542}

\bibitem{diening2023pointwise}
L.~Diening, J.~Rolfes, and A.~J. Salgado, \emph{Pointwise {G}radient {E}stimate
  of the {R}itz {P}rojection}, SIAM J. Numer. Anal. \textbf{62} (2024), no.~3,
  1212--1225. \MR{4748207}

\bibitem{DDO:17}
I.~Drelichman, R.~G. Dur\'{a}n, and I.~Ojea, \emph{A weighted setting for the
  numerical approximation of the {P}oisson problem with singular sources}, SIAM
  J. Numer. Anal. \textbf{58} (2020), no.~1, 590--606. \MR{4060457}

\bibitem{MR1800316}
J.~Duoandikoetxea, \emph{Fourier analysis}, Graduate Studies in Mathematics,
  vol.~29, American Mathematical Society, Providence, RI, 2001, Translated and
  revised from the 1995 Spanish original by David Cruz-Uribe. \MR{1800316}

\bibitem{evans10}
L.~C. Evans, \emph{Partial differential equations}, American Mathematical
  Society, Providence, R.I., 2010.

\bibitem{MR3409135}
L.~C. Evans and R.~F. Gariepy, \emph{Measure theory and fine properties of
  functions}, revised ed., Textbooks in Mathematics, CRC Press, Boca Raton, FL,
  2015. \MR{3409135}

\bibitem{MR0643158}
E.~B. Fabes, C.~E. Kenig, and R.~P. Serapioni, \emph{The local regularity of
  solutions of degenerate elliptic equations}, Comm. Partial Differential
  Equations \textbf{7} (1982), no.~1, 77--116. \MR{643158}

\bibitem{MR1156467}
S.~J. Fromm, \emph{Potential space estimates for {G}reen potentials in convex
  domains}, Proc. Amer. Math. Soc. \textbf{119} (1993), no.~1, 225--233.
  \MR{1156467}

\bibitem{MR0684631}
J.~Garc\'{\i}a-Cuerva, \emph{An extrapolation theorem in the theory of
  {$A\sb{p}$} weights}, Proc. Amer. Math. Soc. \textbf{87} (1983), no.~3,
  422--426. \MR{684631}

\bibitem{MR2491902}
V.~Gol'dshtein and A.~Ukhlov, \emph{Weighted {S}obolev spaces and embedding
  theorems}, Trans. Amer. Math. Soc. \textbf{361} (2009), no.~7, 3829--3850.
  \MR{2491902}

\bibitem{MR3243734}
L.~Grafakos, \emph{Classical {F}ourier analysis}, third ed., Graduate Texts in
  Mathematics, vol. 249, Springer, New York, 2014. \MR{3243734}

\bibitem{MR3396210}
P.~Grisvard, \emph{Elliptic problems in nonsmooth domains}, Classics in Applied
  Mathematics, vol.~69, Society for Industrial and Applied Mathematics (SIAM),
  Philadelphia, PA, 2011, Reprint of the 1985 original [MR0775683].
  \MR{3396210}

\bibitem{MR657523}
M.~Gr\"uter and K.-O. Widman, \emph{The {G}reen function for uniformly elliptic
  equations}, Manuscripta Math. \textbf{37} (1982), no.~3, 303--342.
  \MR{657523}

\bibitem{MR2305115}
J.~Heinonen, T.~Kilpel\"{a}inen, and O.~Martio, \emph{Nonlinear potential
  theory of degenerate elliptic equations}, Dover Publications, Inc., Mineola,
  NY, 2006, Unabridged republication of the 1993 original. \MR{2305115}

\bibitem{JIN2009773}
T.~Jin, V.~Maz'ya, and J.~{Van Schaftingen}, \emph{Pathological solutions to
  elliptic problems in divergence form with continuous coefficients}, Comptes
  Rendus Mathematique \textbf{347} (2009), no.~13, 773--778.

\bibitem{MR0802206}
A.~Kufner, \emph{Weighted {S}obolev spaces}, A Wiley-Interscience Publication,
  John Wiley \& Sons, Inc., New York, 1985, Translated from the Czech.
  \MR{802206}

\bibitem{MENGESHA20112485}
T.~Mengesha and N.~C. Phuc, \emph{Weighted and regularity estimates for
  nonlinear equations on {R}eifenberg flat domains}, Journal of Differential
  Equations \textbf{250} (2011), no.~5, 2485--2507.

\bibitem{MR4408483}
E.~Ot\'{a}rola and A.~J. Salgado, \emph{On the analysis and approximation of
  some models of fluids over weighted spaces on convex polyhedra}, Numer. Math.
  \textbf{151} (2022), no.~1, 185--218. \MR{4408483}

\bibitem{MR3014456}
Tom\'{a}\v{s} Roub\'{\i}\v{c}ek, \emph{Nonlinear partial differential equations
  with applications}, second ed., International Series of Numerical
  Mathematics, vol. 153, Birkh\"{a}user/Springer Basel AG, Basel, 2013.
  \MR{3014456}

\bibitem{MR0745140}
J.~L. Rubio~de Francia, \emph{Factorization theory and {$A\sb{p}$} weights},
  Amer. J. Math. \textbf{106} (1984), no.~3, 533--547. \MR{745140}

\bibitem{ASNSP_1964_3_18_3_385_0}
J.~Serrin, \emph{Pathological solutions of elliptic differential equations},
  Annali della Scuola Normale Superiore di Pisa - Scienze Fisiche e Matematiche
  \textbf{Ser. 3, 18} (1964), no.~3, 385--387 (en). \MR{170094}

\bibitem{MR3169756}
J.~L. Taylor, S.~Kim, and R.~M. Brown, \emph{The {G}reen function for elliptic
  systems in two dimensions}, Comm. Partial Differential Equations \textbf{38}
  (2013), no.~9, 1574--1600. \MR{3169756}

\bibitem{MR1774162}
B.~O. Turesson, \emph{Nonlinear potential theory and weighted {S}obolev
  spaces}, Lecture Notes in Mathematics, vol. 1736, Springer-Verlag, Berlin,
  2000. \MR{1774162}

\bibitem{MR1014685}
W.~P. Ziemer, \emph{Weakly differentiable functions}, Graduate Texts in
  Mathematics, vol. 120, Springer-Verlag, New York, 1989, Sobolev spaces and
  functions of bounded variation. \MR{1014685}

\end{thebibliography}

\end{document}